\documentclass{amsart}
\usepackage{amssymb,color,amsmath,amsfonts,latexsym,amsthm,geometry,graphicx}
\usepackage{amsmath}
\usepackage{amsfonts}
\usepackage{amssymb}
\usepackage{graphicx}

\providecommand{\U}[1]{\protect\rule{.1in}{.1in}}
\providecommand{\U}[1]{\protect\rule{.1in}{.1in}}
\providecommand{\U}[1]{\protect\rule{.1in}{.1in}}
\providecommand{\U}[1]{\protect\rule{.1in}{.1in}}
\newtheorem{theorem}{Theorem}[section]
\newtheorem{corollary}[theorem]{Corollary}
\newtheorem{proposition}[theorem]{Proposition}

\theoremstyle{definition}

\newtheorem{remark}[theorem]{Remark}
\newtheorem{definition}[theorem]{Definition}

\begin{document}
\title[Optimal estimates for summing multilinear operators]{Optimal
estimates for summing multilinear operators}
\author[G. Ara\'{u}jo]{Gustavo Ara\'{u}jo}
\address{Departamento de Matem\'{a}tica \\
Universidade Federal da Para\'{\i}ba \\
58.051-900 - Jo\~{a}o Pessoa, Brazil.}
\email{gdasaraujo@gmail.com}
\author[D. Pellegrino]{Daniel Pellegrino}
\address{Departamento de Matem\'{a}tica \\
Universidade Federal da Para\'{\i}ba \\
58.051-900 - Jo\~{a}o Pessoa, Brazil.}
\email{pellegrino@pq.cnpq.br and dmpellegrino@gmail.com}
\keywords{Absolutely summing operators, Banach spaces}
\thanks{\textit{Mathematics Subject Classification} [2010]: 47B10; 47L22;
46G25}

\begin{abstract}
We show that given a positive integer $m$, a real number $p\in\left[
2,\infty\right)$ and $1\leq s<p^{\ast}$ the set of non--multiple $\left(
r;s\right)$--summing $m$--linear forms on $\ell_{p}\times\cdots\times
\ell_{p}$ contains, except for the null vector, a closed subspace of maximal
dimension whenever $r<\frac{2ms}{s+2m-ms}$. This result is optimal since for
$r\geq\frac{2ms}{s+2m-ms}$ all $m$--linear forms on $\ell_{p}\times
\cdots\times\ell_{p}$ are multiple $\left( r;s\right)$--summing. In
particular, among other results, we generalize a result related to cotype
(from 2010) due to Botelho \textit{et al.}
\end{abstract}

\maketitle

\section{Introduction}

The family of inequalities known as Bohnenblust--Hille, Littlewood's $4/3$
and Hardy--Littlewood (see \cite{bh,hardy,LLL}) dates back to the 30s and,
after a long period of dormancy, have been rediscovered in the recent years
with interesting applications in different fields. In the modern
terminology, these inequalities can be seen as coincidence results in the
theory of multiple summing operators. The main goal of this note is to
investigate in details how the new advances in the study of the
aforementioned inequalities can be explored in the context of multiple
summing operators.

Let $E,E_{1},...,E_{m}$ and $F$ denote Banach spaces over $\mathbb{K}=%
\mathbb{R}$ or $\mathbb{C}$ and let $B_{E^{\ast}}$ denote the closed unit
ball of the topological dual of $E$. If $1\leq q\leq\infty$, the symbol $%
q^{\ast}$ represents the conjugate of $q$. It will be convenient to adopt
that $\frac{c}{\infty}=0$ for any $c>0$. For $s>0$, by $\ell_{s}(E)$ we mean
the space of absolutely $s$--summable sequences in $E$; for $s\geq1$ we
represent by $\ell_{s}^{w}(E)$ the linear space of the sequences $\left(
x_{j}\right) _{j=1}^{\infty}$ in $E$ such that $\left( \varphi\left(
x_{j}\right) \right) _{j=1}^{\infty}\in\ell_{s}$ for every continuous linear
functional $\varphi:E\rightarrow\mathbb{K}$. The function
\begin{equation*}
\left\Vert \left( x_{j}\right) _{j=1}^{\infty}\right\Vert
_{w,s}=\sup_{\varphi\in B_{E^{\ast}}}\left\Vert \left( \varphi\left(
x_{j}\right) \right) _{j=1}^{\infty}\right\Vert _{s}
\end{equation*}
defines a norm on $\ell_{s}^{w}(E)$. The space of all continuous $m$-linear
operators $T:E_{1}\times\cdots\times E_{m}\rightarrow F$, with the $\sup$
norm, is denoted by $\mathcal{L}\left( E_{1},...,E_{m};F\right) $.

The multilinear theory of absolutely summing operators was initiated by
Pietsch \cite{pi1} and nowadays is a very fruitful topic of investigation
(for the linear theory we refer, for instance, to \cite{Di, piet}). The
following concept is a natural extension of the notion of absolutely summing
linear operators to the multilinear setting (see \cite{matos, per}, see also
\cite{campo, popa, popa2, popa5, serrano, rueda} for related and recent approaches):

\begin{definition}
Let $1\leq s\leq r<\infty$. A multilinear operator $T\in\mathcal{L}\left(
E_{1},...,E_{m};F\right) $ is multiple $\left( r;s\right) $--summing if
there exists a $C>0$ such that
\begin{equation*}
\left( \sum_{j_{1},...,j_{m}=1}^{\infty}\left\Vert T(
x_{j_{1}}^{(1)},...,x_{j_{m}}^{(m)}) \right\Vert ^{r}\right) ^{\frac{1}{r}%
}\leq C\prod_{k=1}^{m}\left\Vert ( x_{j}^{(k)}) _{j=1}^{\infty}\right\Vert
_{w,s}
\end{equation*}
for all $( x_{j}^{(k)}) _{j=1}^{\infty}\in\ell_{s}^{w}\left( E_{k}\right) $,
$k\in\{1,...,m\}$. We represent the class of all multiple $\left( r;s\right)
$--summing operators from $E_{1},....,E_{m}$ to $F$ by $\Pi_{\mathrm{mult}%
\left( r;s\right) }\left( E_{1},...,E_{m};F\right) $ and $\pi _{\mathrm{mult}%
\left( r;s\right) }\left( T\right) $ denotes the infimum over all $C$ as
above.
\end{definition}

The classical Bohnenblust--Hille inequality \cite{bh}, in the modern
terminology, can be stated in terms of multiple summing operators, as
remarked in \cite{per} (see also \cite{REMC2010}):

\begin{theorem}[Bohnenblust--Hille]
Every continuous $m$--linear form $T\in\mathcal{L}\left( E_{1},...,E_{m};%
\mathbb{K}\right) $ is multiple $\left( \frac {2m}{m+1};1\right) $--summing
for all Banach spaces $E_{1},...,E_{m}$ and $\frac{2m}{m+1}$ is optimal.
\end{theorem}

The paper is organized as follows. In Section 2, we prove our main result,
which, in particular, shows that if $1<s<p^{\ast }$, the set $\mathcal{L}%
\left( ^{m}\ell _{p};\mathbb{K}\right) \smallsetminus \Pi _{\mathrm{mult}%
\left( \frac{2m}{m+1};s\right) }\left( ^{m}\ell _{p};\mathbb{K}\right) $
contains, except for the null vector, a closed infinite--dimensional Banach
space with the same dimension of $\mathcal{L}(^{m}\ell _{p};\mathbb{K})$. In
Section 3, we show some consequences of the result of the previous section.
For instance, as a particular case of our main result, we observe a new
optimality component of the Bohnenblust--Hille inequality: the term $1$ from
the pair $\left( \frac{2m}{m+1};1\right) $ is also optimal. In Sections 4
and 5, we investigate the optimality of coincidence results for multiple
summing operators in $c_{0}$ and in the framework of absolutely summing
multilinear operators, respectively.

\section{Maximal subspaces and multiple summability}

For a given Banach space $E$, a subset $A\subset E$ is \textit{spaceable} if
$A\cup\{0\}$ contains a closed infinite--dimensional subspace $V$ of $E$
(for details on spaceability and the related notion of lineability we refer
to \cite{bernal} and the references therein). When $\dim V=\dim E$, $A$ is
called \textit{maximal spaceable}. From now on $\mathfrak{c}$ denotes the
cardinality of the continuum.

\begin{theorem}
\label{xx} Let $m\geq1,$ $p\in\left[ 2,\infty\right) .$ If $1\leq s<p^{\ast}
$ and
\begin{equation*}
r<\frac{2ms}{s+2m-ms}
\end{equation*}
then
\begin{equation*}
\mathcal{L}\left( ^{m}\ell_{p};\mathbb{K}\right) \smallsetminus \Pi_{\mathrm{%
mult}\left( r;s\right) } \left( ^{m}\ell_{p};\mathbb{K}\right)
\end{equation*}
is maximal spaceable.
\end{theorem}

\begin{proof}
We consider the case of complex scalars. The case of real scalars is
obtained from the complex case via a standard complexification argument (see
\cite{REMC2010}). An extended version of the Kahane--Salem--Zygmund
inequality (see \cite[Lemma 6.1]{alb} and also \cite{bay} for several
related results) asserts that if $m,n\geq1$ and $p\in\lbrack2,\infty]$, then
there exists a $m$--linear map $A_{n}:\ell_{p}^{n}\times\cdots\times
\ell_{p}^{n}\rightarrow\mathbb{K}$ of the form
\begin{equation}
A_{n}(z^{(1)},\dots,z^{(m)})=\sum_{j_{1},\dots,j_{m}=1}^{n}\pm
z_{j_{1}}^{(1)}\cdots z_{j_{m}}^{(m)}  \label{1}
\end{equation}
such that
\begin{equation*}
\Vert A_{n}\Vert\leq C_{m}n^{\frac{m+1}{2}-\frac{m}{p}}
\end{equation*}
for some constant $C_{m}>0$.

Let
\begin{equation*}
\beta:=\frac{p+s-ps}{ps}.
\end{equation*}
Observe that $s<p^{\ast}$ implies $\beta>0$. We have
\begin{equation*}
\left( \sum\limits_{j_{1},...,j_{m}=1}^{n}\left\vert A_{n}\left( \frac{%
e_{j_{1}}}{j_{1}^{\beta}},...,\frac{e_{j_{m}}}{j_{m}^{\beta}}\right)
\right\vert ^{r}\right) ^{\frac{1}{r}}\leq\pi_{\mathrm{mult}\left(
r;s\right) }\left( A_{n}\right) \left\Vert \left( \frac{e_{j}}{j^{\beta} }%
\right) _{j=1}^{n}\right\Vert _{w,s}^{m}
\end{equation*}
i.e.,
\begin{equation*}
\left( \sum\limits_{j_{1},...,j_{m}=1}^{n}\left\vert \frac{1}{j_{1}^{\beta
}...j_{m}^{\beta}}\right\vert ^{r}\right) ^{\frac{1}{r}}\leq\pi _{\mathrm{%
mult}\left( r;s\right) }\left( A_{n}\right) \left\Vert \left( \frac{e_{j}}{%
j^{\beta}}\right) _{j=1}^{n}\right\Vert _{w,s}^{m}.
\end{equation*}
But, for $n\geq2,$ since $\frac{1}{\frac{1}{\beta s}}+\frac{1}{\frac{%
p^{\ast} }{s}}=1,$ we obtain
\begin{align*}
\left\Vert \left( \frac{e_{j}}{j^{\beta}}\right) _{j=1}^{n}\right\Vert
_{w,s} & =\sup_{\varphi\in B_{(\ell_{p}^{n})^{\ast}}}\left( \sum
\limits_{j=1}^{n}\left\vert \varphi\left( \frac{e_{j}}{j^{\beta}}\right)
\right\vert ^{s}\right) ^{\frac{1}{s}} \\
& =\sup_{\varphi\in B_{\ell_{p^{\ast}}^{n}}}\left( \sum\limits_{j=1}
^{n}\left\vert \varphi_{j}\right\vert ^{s}\frac{1}{j^{\beta s}}\right) ^{%
\frac{1}{s}} \\
& \leq\left( \left( \sum\limits_{j=1}^{n}\left\vert \varphi_{j}\right\vert
^{p^{\ast}}\right) ^{\frac{s}{p^{\ast}}}\left( \sum\limits_{j=1}^{n}\frac {1%
}{j}\right) ^{\beta s}\right) ^{\frac{1}{s}} \\
& <\left( 1+\log n\right) ^{\beta}.
\end{align*}
Hence%
\begin{equation*}
\left( \sum\limits_{j=1}^{n}\frac{1}{j^{r\beta}}\right) ^{\frac{m}{r}} <\pi_{%
\mathrm{mult}\left( r;s\right) }\left( A_{n}\right) \left( 1+\log n\right)
^{m\beta}
\end{equation*}
and consequently
\begin{equation*}
\left( n^{1-r\beta}\right) ^{\frac{m}{r}}<\pi_{\mathrm{mult}\left(
r;s\right) }\left( A_{n}\right) \left( 1+\log n\right) ^{m\beta}.
\end{equation*}

Since $\left\Vert A_{n}\right\Vert \leq C_{m}n^{\frac{m+1}{2}-\frac{m}{p}}$
we have
\begin{equation*}
\frac{\pi_{\mathrm{mult}\left( r;s\right) }\left( A_{n}\right) }{\left\Vert
A_{n}\right\Vert }>\frac{n^{\frac{m}{r}-\left( \frac{p+s-ps}{ps}\right) m}}{%
\left( 1+\log n\right) ^{m\beta}C_{m}n^{\frac{m+1}{2}-\frac{m}{p}}}.
\end{equation*}
By making $n\rightarrow\infty$ and using that $r<\frac{2ms}{s+2m-ms}$ we get
\begin{equation*}
\lim_{n\rightarrow\infty}\frac{\pi_{\mathrm{mult}\left( r;s\right) }\left(
A_{n}\right) }{\left\Vert A_{n}\right\Vert }=\infty
\end{equation*}
and from the Open Mapping Theorem we conclude that $\Pi_{\mathrm{mult}\left(
r;s\right) }\left( ^{m}\ell_{p};\mathbb{K}\right) $ is not closed in $%
\mathcal{L}\left( ^{m}\ell_{p};\mathbb{K}\right) .$ From \cite[Theorem 5.6
and its reformulation]{drew} (see also \cite{KT}) we conclude that $\mathcal{%
L}\left( ^{m}\ell_{p};\mathbb{K}\right) \smallsetminus \Pi_{\mathrm{mult}%
\left( r;s\right) }\left( ^{m}\ell_{p};\mathbb{K}\right) $ is spaceable.

It remains to prove the maximal spaceability. Since $\mathcal{L}(^{m}\ell
_{p};\mathbb{K})$ is a Banach space, we have
\begin{equation*}
\dim(\mathcal{L}(^{m}\ell _{p};\mathbb{K}))\geq\mathfrak{c}.
\end{equation*}
Let $\gamma$ be a Hamel basis of $\mathcal{L}(^{m}\ell_{p};\mathbb{K})$ and
\begin{equation*}
\begin{array}{ccccc}
g & : & \gamma & \rightarrow & \mathbb{K}^{c_{00}(\mathbb{Q})\times
\cdots\times c_{00}(\mathbb{Q})} \\
&  & T & \mapsto & T|_{c_{00}(\mathbb{Q})\times\cdots\times c_{00}(\mathbb{Q}%
)},%
\end{array}%
\end{equation*}
where $\mathbb{K}^{c_{00}(\mathbb{Q})\times\cdots\times c_{00}(\mathbb{Q})}$
is the set of all functions from $c_{00}(\mathbb{Q})\times\cdots\times
c_{00}(\mathbb{Q})$ to $\mathbb{K}$ (by $c_{00}(\mathbb{Q})$ we mean the
eventually null sequences with rational entries). From the density of $%
c_{00}(\mathbb{Q})$ in $\ell_{p}$ we conclude that $g$ is injective and so
\begin{equation*}
\dim\left( \mathcal{L}(^{m}\ell_{p};\mathbb{K})\right) =\mathrm{card}\left(
\gamma\right) \leq\mathrm{card}\left( \mathbb{K}^{c_{00}(\mathbb{Q}%
)\times\cdots\times c_{00}(\mathbb{Q})}\right) =\mathfrak{c}.
\end{equation*}

Therefore, if
\begin{equation*}
V\subseteq(\mathcal{L}(^{m}\ell_{p};\mathbb{K})\smallsetminus\Pi _{\mathrm{%
mult}(r;s)}(^{m}\ell_{p};\mathbb{K}))\cup\{0\}
\end{equation*}
is a closed infinite--dimensional subspace of $\mathcal{L}(^{m}\ell _{p};%
\mathbb{K})$, we have $\dim(V)\leq\mathfrak{c}$. Since $V$ is a Banach
space, we also have $\dim(V)\geq\mathfrak{c}$. Thus, by the
Cantor-Bernstein-Schroeder Theorem, it follows that $\dim(V)=\mathfrak{c}$
and the proof is done.
\end{proof}

\begin{remark}
Note that it was not necessary to suppose the Continuum Hypothesis. In fact,
for instance the proof given in \cite[Remark 2.5]{studia} of the fact that
the dimension of every infinite--dimensional Banach space is, at least, $%
\mathfrak{c}$ does not depends on the Continuum Hypothesis.
\end{remark}

\section{Some consequences}

The following result is a simple consequence of Theorem \ref{xx}.

\begin{corollary}
\label{corol} Let $m\geq2$ and $r\in\left[ \frac{2m}{m+1},2\right] $. Then
\begin{equation*}
\sup\left\{ s:\mathcal{L}(^{m}\ell_{p};\mathbb{K})=\Pi_{\mathrm{mult}\left(
r;s\right) }(^{m}\ell_{p};\mathbb{K})\right\} \leq\frac{2mr}{mr+2m-r}
\end{equation*}
for all $2\leq p<\frac{2mr}{r+mr-2m}$.
\end{corollary}

\begin{proof}
Since $\frac{2m}{m+1}\leq r\leq2<2m$, it follows that $1\leq\frac {2mr}{%
mr+2m-r}$ and $2<\frac{2mr}{r+mr-2m}$. Note that
\begin{equation*}
s>\frac{2mr}{mr+2m-r}
\end{equation*}
implies
\begin{equation*}
r<\frac{2ms}{s+2m-ms}.
\end{equation*}
Therefore, for $2\leq p<\frac{2mr}{r+mr-2m}$, from Theorem \ref{xx} we know
that
\begin{equation*}
\mathcal{L}\left( ^{m}\ell_{p};\mathbb{K}\right) \smallsetminus \Pi_{\mathrm{%
mult}\left( r;s\right) }\left( ^{m}\ell_{p};\mathbb{K}\right)
\end{equation*}
is spaceable for all $\frac{2mr}{mr+2m-r}<s<p^{\ast}$ (note that $p<\frac {%
2mr}{r+mr-2m}$ implies $p^{\ast}>\frac{2mr}{mr+2m-r}$). In particular, for $%
2\leq p<\frac{2mr}{r+mr-2m}$,
\begin{equation*}
\sup\left\{ s:\mathcal{L}(^{m}\ell_{p};\mathbb{K})= \Pi_{\mathrm{mult}\left(
r;s\right) }(^{m}\ell_{p};\mathbb{K})\right\} \leq\frac{2mr}{mr+2m-r}.
\end{equation*}
\end{proof}

This corollary together with our main result (Theorem \ref{xx}) ensures
that, for $r\in \left[ \frac{2m}{m+1},2\right] $ and $2\leq p<\frac{2mr}{%
r+mr-2m}$,
\begin{equation*}
\sup \left\{ s:\mathcal{L}(^{m}\ell _{p};\mathbb{K})=\Pi _{\mathrm{mult}%
\left( r;s\right) }(^{m}\ell _{p};\mathbb{K})\right\} =\frac{2mr}{mr+2m-r}.
\end{equation*}%
When $p=2$ the expression above recovers the optimality of \cite[Theorem 5.14%
]{REMC2010} in the case of $m$--linear forms on $\ell _{2}\times \cdots
\times \ell _{2}$.

We recall that a Banach space $X$ has cotype $2\leq q<\infty $ if there is a
constant $C>0$ such that
\begin{equation}
\left( \sum_{k=1}^{n}\Vert x_{k}\Vert ^{q}\right) ^{\frac{1}{q}}\leq C\left(
\int_{[0,1]}\left\Vert \sum_{k=1}^{n}r_{k}(t)x_{k}\right\Vert ^{2}dt\right)
^{\frac{1}{2}}  \label{qqq}
\end{equation}%
for all positive integers $n$ and all $x_{1},\dots ,x_{n}\in X$, where each $%
r_{k}$ denotes the $k$-th Rademacher function. The smallest of all these
constants is denoted by $C_{q}(X)$ and it is called the cotype $q$ constant
of $X$. For details and classical results we refer to \cite{Di, pisier2}

In 2010 G. Botelho, C. Michels and D. Pellegrino \cite{REMC2010} have shown
that for $m\geq1$ and Banach spaces $E_{1},...,E_{m}$ of cotype $2$,
\begin{equation*}
\mathcal{L}\left( E_{1},...,E_{m};\mathbb{K}\right) = \Pi_{\mathrm{mult}%
\left( 2;\frac{2m}{2m-1}\right) }\left( E_{1},...,E_{m};\mathbb{K}\right)
\end{equation*}
and for Banach spaces of cotype $k>2$,
\begin{equation*}
\mathcal{L}\left( E_{1},...,E_{m};\mathbb{K}\right) =\Pi_{\mathrm{mult}%
\left( 2;\frac{km}{km-1}-\epsilon\right) }\left( E_{1},...,E_{m};\mathbb{K}%
\right)
\end{equation*}
for all sufficiently small $\epsilon>0$.

We now remark that it is not necessary to make any assumptions on the Banach
spaces $E_{1},...,E_{m}$ and $\frac{2m}{2m-1}$ holds in all cases. Given $%
k>2 $, in \cite[page 194]{laa} it is said that it is not known if $s=\frac {%
km}{km-1}$ is attained or not in
\begin{equation*}
\sup\{s:\mathcal{L}(E_{1},...,E_{m};\mathbb{K})= \Pi_{\mathrm{mult}%
(2;s)}(E_{1},...,E_{m};\mathbb{K})\text{ for all }E_{j}\text{ of cotype }%
k\}\geq\frac{km}{km-1}.
\end{equation*}
The fact that $\frac{2m}{2m-1}$ can replace $\frac{km}{km-1}$ in all cases
ensures that $s=\frac{km}{km-1}$ is not attained and thus improves the
estimate of \cite[Corollary 3.1]{laa}, which can be improved to
\begin{equation*}
\sup\{s:\mathcal{L}(E_{1},...,E_{m};\mathbb{K})=\Pi_{\mathrm{mult}%
(2;s)}(E_{1},...,E_{m};\mathbb{K})\text{ for all }E_{j}\text{ of cotype }%
k\}\in\left[ \frac{2m}{2m-1},\frac{2km}{2km+k-2m}\right]
\end{equation*}
if $k>2$ and $m\geq k$ is a positive integer.

More precisely we prove the following more general result. We remark that
the part (i) of the above theorem can be also derived from \cite{n,dimant},
although it is not explicitly written in the aforementioned papers:

\begin{theorem}
\label{000} Let $m\geq2$ and let $r\in\left[ \frac{2m}{m+1},\infty\right)$.
Then the optimal $s$ such that
\begin{equation*}
\mathcal{L}\left( E_{1},...,E_{m};\mathbb{K}\right) = \Pi_{\mathrm{mult}%
\left( r;s\right) }\left( E_{1},...,E_{m};\mathbb{K}\right) .
\end{equation*}
for all Banach spaces $E_{1},...,E_{m}$ is:

\begin{itemize}
\item[(i)] $\frac{2mr}{mr+2m-r}$ if $r\in\left[ \frac{2m}{m+1},2\right]$;

\item[(ii)] $\frac{mr}{mr+1-r}$ if $r\in\left( 2,\infty\right)$.
\end{itemize}
\end{theorem}

\begin{proof}
Proof of (i). For $q\geq 1$, let $X_{q}=\ell _{q}$ and let us define $%
X_{\infty }=c_{0}$. Let
\begin{equation*}
q:=\frac{2mr}{r+mr-2m}.
\end{equation*}%
Since $r\in \left[ \frac{2m}{m+1},2\right] $ we have that $q\in \left[
2m,\infty \right] $. Since
\begin{equation*}
\frac{m}{q}\leq \frac{1}{2}\qquad \mbox{and}\qquad r=\frac{2m}{m+1-\frac{2m}{%
q}},
\end{equation*}%
from the multilinear Hardy--Littlewood inequality (see, for instance, \cite%
{alb,hardy,pra}) there is a constant $C\geq 1$ such that
\begin{equation*}
\left( \sum\limits_{j_{1},....,j_{m}=1}^{\infty }\left\vert A\left(
e_{j_{1}},...,e_{j_{m}}\right) \right\vert ^{r}\right) ^{\frac{1}{r}}\leq
C\left\Vert A\right\Vert
\end{equation*}%
for all continuous $m$--linear forms $A:X_{q}\times \cdots \times
X_{q}\rightarrow \mathbb{K}$. Let $T\in \mathcal{L}\left( E_{1},...,E_{m};%
\mathbb{K}\right) $ and $(x_{j}^{(k)})_{j=1}^{\infty }\in \ell _{q^{\ast
}}^{w}(E_{k})$, $k=1,...,m$. Now we use a standard argument (see \cite{n})
to lift the result from $X_{q}$ to arbitrary Banach spaces. From \cite[%
Proposition 2.2]{Di} we know that exist a continuous linear operator $%
u_{k}:X_{q}\rightarrow E_{k}$ so that $u_{k}\cdot e_{j_{k}}=x_{j_{k}}^{(k)}$
and
\begin{equation*}
\left\Vert u_{k}\right\Vert =\left\Vert (x_{j}^{(k)})_{j=1}^{\infty
}\right\Vert _{w,q^{\ast }}
\end{equation*}%
for all $k=1,...,m$. Therefore, $S:X_{q}\times \cdots \times
X_{q}\rightarrow \mathbb{K}$ defined by $S(y_{1},...,y_{m})=T\left(
u_{1}\cdot y_{1},...,u_{m}\cdot y_{m}\right) $ is $m$--linear, continuous
and
\begin{equation*}
\Vert S\Vert \leq \Vert T\Vert \prod_{k=1}^{m}\Vert u_{k}\Vert
=\prod_{k=1}^{m}\left\Vert (x_{j}^{(k)})_{j=1}^{\infty }\right\Vert
_{w,q^{\ast }}.
\end{equation*}%
Hence
\begin{equation*}
\left( \sum\limits_{j_{1},...,j_{m}=1}^{\infty }\left\vert T\left(
x_{j_{1}}^{(1)},...,x_{j_{m}}^{(m)}\right) \right\vert ^{r}\right) ^{\frac{1%
}{r}}\leq C\Vert T\Vert \prod_{k=1}^{m}\left\Vert
(x_{j}^{(k)})_{j=1}^{\infty }\right\Vert _{w,q^{\ast }},
\end{equation*}%
and, as $q^{\ast }=\frac{2mr}{mr+2m-r}$, the last inequality proves that,
for all $m\geq 2$ and $r\in \left[ \frac{2m}{m+1},2\right] $, we have
\begin{equation*}
\mathcal{L}\left( E_{1},...,E_{m};\mathbb{K}\right) =\Pi _{\mathrm{mult}%
\left( r;\frac{2mr}{mr+2m-r}\right) }\left( E_{1},...,E_{m};\mathbb{K}%
\right) .
\end{equation*}%
Now let us prove the optimality. From what we have just proved, for $r\in %
\left[ \frac{2m}{m+1},2\right] $, we have
\begin{eqnarray*}
U_{m,r}:= &&\sup \left\{ s:\mathcal{L}(E_{1},...,E_{m};\mathbb{K})=\Pi _{%
\mathrm{mult}\left( r;s\right) }(E_{1},...,E_{m};\mathbb{K})\text{ for all
Banach spaces }E_{j}\right\}  \\
&\geq &\frac{2mr}{mr+2m-r}.
\end{eqnarray*}%
From Corollary \ref{corol} we have, for $2\leq p<\frac{2mr}{r+mr-2m}$,
\begin{equation*}
\sup \left\{ s:\mathcal{L}(^{m}\ell _{p};\mathbb{K})=\Pi _{\mathrm{mult}%
\left( r;s\right) }(^{m}\ell _{p};\mathbb{K})\right\} \leq \frac{2mr}{mr+2m-r%
}.
\end{equation*}%
Therefore,
\begin{equation*}
U_{m,r}\leq \sup \left\{ s:\mathcal{L}(^{m}\ell _{p};\mathbb{K})=\Pi _{%
\mathrm{mult}\left( r;s\right) }(^{m}\ell _{p};\mathbb{K})\right\} \leq
\frac{2mr}{mr+2m-r}
\end{equation*}%
and we conclude that $U_{m,r}=\frac{2mr}{mr+2m-r}$.

Proof of (ii). Given $r>2$ consider $m<p<2m$ such that $r=\frac{p}{p-m}$. In
this case, $p=\frac{mr}{r-1}$ and $p^{\ast}=\frac{mr}{mr+1-r}$. From \cite%
{dimant} we know that
\begin{equation}
\Pi_{\mathrm{mult}(\frac{p}{p-m};p^{\ast})}(E_{1},...,E_{m};\mathbb{K})=%
\mathcal{L}(E_{1},...,E_{m};\mathbb{K})  \label{9k}
\end{equation}
for all Banach spaces $E_{1},...,E_{m}$, i.e.,
\begin{equation*}
\Pi_{\mathrm{mult}(r;\frac{mr}{mr+1-r})}(E_{1},...,E_{m};\mathbb{K})=%
\mathcal{L}(E_{1},...,E_{m};\mathbb{K})
\end{equation*}
for all Banach spaces $E_{1},...,E_{m}$. Also, for $E_{j}=\ell_{p}$ for all $%
j$ we know that
\begin{equation}
\Pi_{\mathrm{mult}(\frac{p}{p-m};p^{\ast})}(\ell_{p},...,\ell_{p};\mathbb{K}%
)=\mathcal{L}(\ell_{p},...,\ell_{p};\mathbb{K})  \label{in}
\end{equation}
is optimal, i.e., $\frac{p}{p-m}$ cannot be improved. If $s>p^{\ast}$ let $%
q^{\ast}=s$ and then $q<p$ (we can always suppose $s$ close to $p^{\ast}$
and thus $m<q<2m$)$.$ From (\ref{9k}) we have
\begin{equation*}
\Pi_{\mathrm{mult}(\frac{q}{q-m};q^{\ast})}(E_{1},...,E_{m};\mathbb{K})=%
\mathcal{L}(E_{1},...,E_{m};\mathbb{K})
\end{equation*}
and from (\ref{in}) in the case of $\ell_{q}$ instead of $\ell_{p}$, we have
\begin{equation*}
\Pi_{\mathrm{mult}(\frac{q}{q-m};q^{\ast})}(\ell_{q},...,\ell_{q};\mathbb{K}%
)=\mathcal{L}(\ell_{q},...,\ell_{q};\mathbb{K})
\end{equation*}
and $\frac{q}{q-m}$ is optimal. Since $\frac{q}{q-m}>\frac{p}{p-m}$ we
conclude that%
\begin{equation*}
\Pi_{\mathrm{mult}(\frac{p}{p-m};q^{\ast})}(\ell_{q},...,\ell_{q};\mathbb{K}%
)\neq\mathcal{L}(\ell_{q},...,\ell_{q};\mathbb{K}),
\end{equation*}
i.e.,%
\begin{equation*}
\Pi_{\mathrm{mult}(\frac{p}{p-m};s)}(\ell_{q},...,\ell_{q};\mathbb{K})\neq%
\mathcal{L}(\ell_{q},...,\ell_{q};\mathbb{K}).
\end{equation*}
and the proof is done.
\end{proof}

\begin{figure}[h]
\centering
\includegraphics[scale=0.45]{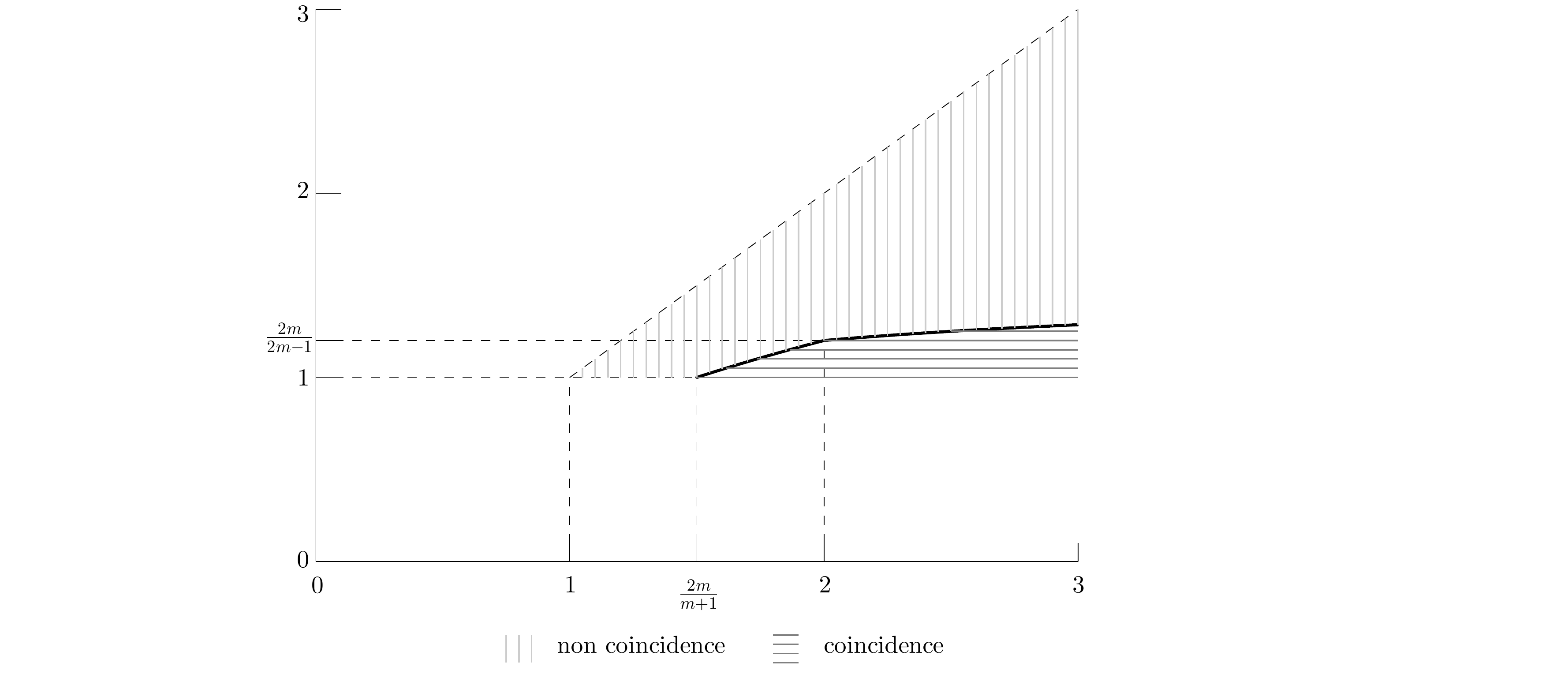}
\caption{Coincidence zone for $\Pi_{\mathrm{mult}(r;s)}(E_1,...,E_m;\mathbb{K%
}),(r,s)\in[1,\infty)\times[1,r]$.}
\label{figure1}
\end{figure}

\bigskip

The table below details the results of coincidence and non coincidence in
the ``boundaries" of Figure \ref{figure1}. We can clearly see that the only
case that remains open is the case $(r;s)$ with $r>2$ and $\frac{2m}{2m-1}<
s\leq\frac{mr}{mr+1-r}$.

\bigskip

\begin{center}
\begin{tabular}{|c|c|c|}
\hline
$r\geq1$ & $s=r$ & non coincidence \\ \hline
$1\leq r < \frac{2m}{m+1}$ & $s = 1$ & non coincidence \\ \hline
$\frac{2m}{m+1}\leq r \leq2$ & $s=\frac{2mr}{mr+2m-r}$ & coincidence \\
\hline
$r\geq\frac{2m}{m+1} $ & $s=1$ & coincidence \\ \hline
$r> 2 $ & $s=\frac{mr}{mr+1-r}$ & coincidence \\ \hline
\end{tabular}
\end{center}

\bigskip

\section{Multiple $\left( r;s\right) $--summing forms in $c_{0}$ and $%
\ell_{\infty}$ spaces}

From standard localization procedures, coincidence results for $c_{0}$ and $%
\ell_{\infty}$ are the same; so we will restrict our attention to $c_{0}.$
It is well known that $\Pi_{\mathrm{mult}\left( r;s\right) }\left( ^{m}c_{0};%
\mathbb{K}\right) =\mathcal{L}\left( ^{m}c_{0};\mathbb{K}\right) $ whenever $%
r\geq s\geq2$ (see \cite{REMC2010})$.$ When $s=1,$ as a consequence of the
Bohnenblust--Hille inequality, we also know that the equality holds if and
only if $s\geq\frac{2m}{m+1}$. The next result encompasses essentially all
possible cases:

\begin{proposition}
If $s\in\left[ 1,\infty\right) $ then
\begin{equation*}
\inf\left\{ r:\Pi_{\mathrm{mult}\left( r;s\right) }\left( ^{m}c_{0};\mathbb{K%
}\right) =\mathcal{L}\left( ^{m}c_{0};\mathbb{K}\right) \right\} =\left\{
\begin{array}{lll}
\dfrac{2m}{m+1} & \text{if} & 1\leq s\leq\dfrac{2m}{m+1}\vspace{0.2cm}, \\
s & \text{if} & s\geq\dfrac{2m}{m+1}.%
\end{array}
\right.
\end{equation*}
\end{proposition}

\begin{proof}
The case $r\geq s\geq 2$ is immediate (see \cite[Corollary 4.10]{REMC2010}).
The Bohnenblust--Hille inequality assures that when $s=1$ the best choice
for $r$ is $\frac{2m}{m+1}$. So, it is obvious that for $1\leq s\leq \frac{2m%
}{m+1}$ the best value for $r$ is not smaller than $\frac{2m}{m+1}.$ More
precisely,%
\begin{equation*}
\Pi _{\mathrm{mult}\left( r;s\right) }\left( ^{m}c_{0};\mathbb{K}\right)
\neq \mathcal{L}\left( ^{m}c_{0};\mathbb{K}\right)
\end{equation*}%
whenever $\left( r,s\right) \in \left[ 1,\frac{2m}{m+1}\right) \times \left[
1,\frac{2m}{m+1}\right] $ and $r\geq s$. An adaptation of deep result due to
Pisier (\cite{pp}) to multiple summing operators (see \cite[Theorem 3.16]%
{arc} or \cite[Lemma 5.2]{REMC2010}) combined with the coincidence result
for $\left( r;s\right) =\left( \frac{2m}{m+1};1\right) $ tells us that we
also have coincidences for $\left( \frac{2m}{m+1};s\right) $ for all $1<s<%
\frac{2m}{m+1}.$ The remaining case $\left( r;s\right) $ with $\frac{2m}{m+1}%
<s<2$ follows from an interpolation procedure in the lines of \cite{REMC2010}%
. More precisely, given $\frac{2m}{m+1}<r<2$ and $0<\delta <\frac{r(2-\theta
)-2}{2-\theta }$, where $\theta =\frac{mr+r-2m}{r}$, consider%
\begin{equation*}
\epsilon =\frac{2m}{m+1}-\frac{2(1-\theta )(r-\delta )}{2-\theta (r-\delta )}%
.
\end{equation*}%
Since $1<\frac{2m}{m+1}-\epsilon <\frac{2m}{m+1}$, we know that $\mathcal{L}%
(^{m}c_{0};\mathbb{K})=\Pi _{\mathrm{mult}\left( \frac{2m}{m+1};\frac{2m}{m+1%
}-\epsilon \right) }(^{m}c_{0};\mathbb{K})$. Since $\mathcal{L}(^{m}c_{0};%
\mathbb{K})=\Pi _{\mathrm{mult}\left( 2;2\right) }(^{m}c_{0};\mathbb{K})$
and
\begin{equation*}
\dfrac{1}{r}=\dfrac{\theta }{2}+\dfrac{1-\theta }{\frac{2m}{m+1}}\qquad
\text{and}\qquad \dfrac{1}{r-\delta }=\dfrac{\theta }{2}+\dfrac{1-\theta }{%
\frac{2m}{m+1}-\epsilon },
\end{equation*}%
by interpolation we conclude $\mathcal{L}(^{m}c_{0};\mathbb{K})=\Pi _{%
\mathrm{mult}\left( r;r-\delta \right) }(^{m}c_{0};\mathbb{K})$.
\end{proof}

\bigskip

\begin{figure}[h]
\centering
\includegraphics[scale=0.45]{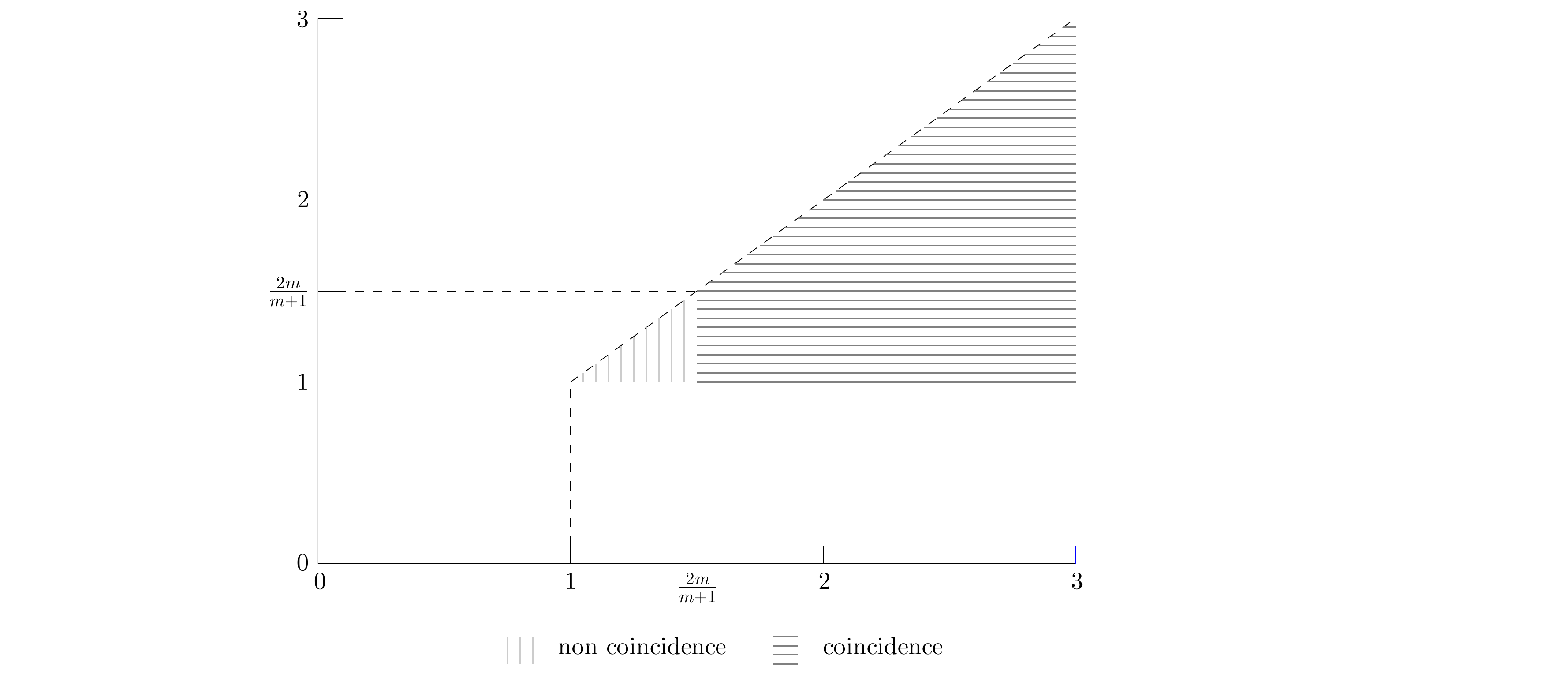}
\caption{Coincidence zone for $\Pi_{\mathrm{mult}(r;s)}(^m c_0;\mathbb{K}%
),(r,s)\in[1,\infty)\times[1,r]$.}
\label{figure2}
\end{figure}

\bigskip

The table below details the results of coincidence and non coincidence in
the ``boundaries" of Figure \ref{figure2}.

\bigskip

\begin{center}
\begin{tabular}{|c|c|c|}
\hline
$1\leq r < \frac{2m}{m+1}$ & $s = 1$ & non coincidence \\ \hline
$r=\frac{2m}{m+1}$ & $1\leq s<\frac{2m}{m+1}$ & coincidence \\ \hline
$r\geq\frac{2m}{m+1} $ & $s=1$ & coincidence \\ \hline
$1\leq r<\frac{2m}{m+1}$ & $s=r$ & non coincidence \\ \hline
$\frac{2m}{m+1}\leq r< 2 $ & $s=r$ & not known \\ \hline
$r\geq2 $ & $s=r$ & coincidence \\ \hline
\end{tabular}
\end{center}

\bigskip

We can see that the only case that remains open is the case $(r;s)$ with $%
\frac{2m}{m+1}\leq r< 2$ and $s=r$.

\section{Absolutely summing multilinear operators}

For $1\leq s<\infty$ and $r\geq\frac{s}{m}$ recall that a continuous $m$%
--linear operator $A:E_{1}\times\cdots\times E_{m}\rightarrow F$ is
absolutely $\left( r;s\right) $--summing if there is a $C>0$ such that%
\begin{equation*}
\left( \displaystyle\sum\limits_{j=1}^{n} \left\Vert A(
x_{j}^{(1)},...,x_{j}^{(m)}) \right\Vert ^{r}\right) ^{\frac{1}{r}}\leq C %
\displaystyle\prod\limits_{k=1}^{m} \sup_{\varphi\in B_{E_{k}^{\ast}}}\left( %
\displaystyle\sum\limits_{j=1}^{n} \left\vert \varphi( x_{j}^{(k)})
\right\vert ^{s}\right) ^{\frac{1}{s}}
\end{equation*}
for all positive integers $n$ and all $( x_{j}^{(k)}) _{j=1}^{n}\in E_{k}$, $%
k=1,...,m$. We represent the class of all absolutely $\left( r;s\right) $%
--summing multiple operators from $E_{1},....,E_{m}$ to $F$ by $\Pi _{%
\mathrm{as}\left( r;s\right) }\left( E_{1},...,E_{m};F\right) $ and $\pi_{%
\mathrm{as}\left( r;s\right) }\left( T\right) $ denotes the infimum over all
$C$ as above.

Combining the Defant--Voigt Theorem (first stated and proved in \cite[%
Theorem 3.10]{alencarmatos}; see also, e.g., \cite[Theorem 3]{aromlacruz}
(for complex scalars) or \cite[Corollary 3.2]{portmath}) and a canonical
inclusion theorem (see \cite{indag,matos2}) we conclude that, for $r,s\geq1$
and $s\leq\frac{mr}{mr+1-r},$ we have
\begin{equation*}
\Pi_{\mathrm{as}(r;s)}(E_{1},...,E_{m};\mathbb{K})=\mathcal{L}%
(E_{1},...,E_{m};\mathbb{K})
\end{equation*}
for all $E_{1},....,E_{m}.$

From \cite[Proposition 1]{zal} it is possible to prove that for $r>1$ and $%
\frac{r}{mr+1-r}\leq t<r$,
\begin{equation*}
\Pi _{\mathrm{as}(t;\frac{mr}{mr+1-r})}(E_{1},...,E_{m};\mathbb{K})\neq
\mathcal{L}(E_{1},...,E_{m};\mathbb{K})
\end{equation*}%
for some choices of $E_{1},...,E_{m}$. In fact, given $r>1$, consider $p>m$
such that $\frac{p}{p-m}=r$ and observe that in this case $\frac{mr}{mr+1-r}%
=p^{\ast }$ and thus we just need to prove that for all $\frac{p^{\ast }}{m}%
\leq t<\frac{p}{p-m}$, we have
\begin{equation*}
\Pi _{\mathrm{as}(t;p^{\ast })}(E_{1},...,E_{m};\mathbb{K})\neq \mathcal{L}%
(E_{1},...,E_{m};\mathbb{K}).
\end{equation*}%
From \cite[Proposition 1]{zal} we know that if $p>m$ and $\frac{p^{\ast }}{m}%
\leq t<\frac{p}{p-m}$, then there is a continuous $m$--linear form $\phi $
such that
\begin{equation*}
\phi \notin \Pi _{\mathrm{as}(t;p^{\ast })}(E_{1},...,E_{m};\mathbb{K}),
\end{equation*}%
i.e.,
\begin{equation*}
\Pi _{\mathrm{as}(t;p^{\ast })}(E_{1},...,E_{m};\mathbb{K})\neq \mathcal{L}%
(E_{1},...,E_{m};\mathbb{K}).
\end{equation*}

All these information together give us the following figure:

\bigskip

\begin{figure}[h]
\centering
\includegraphics[scale=0.45]{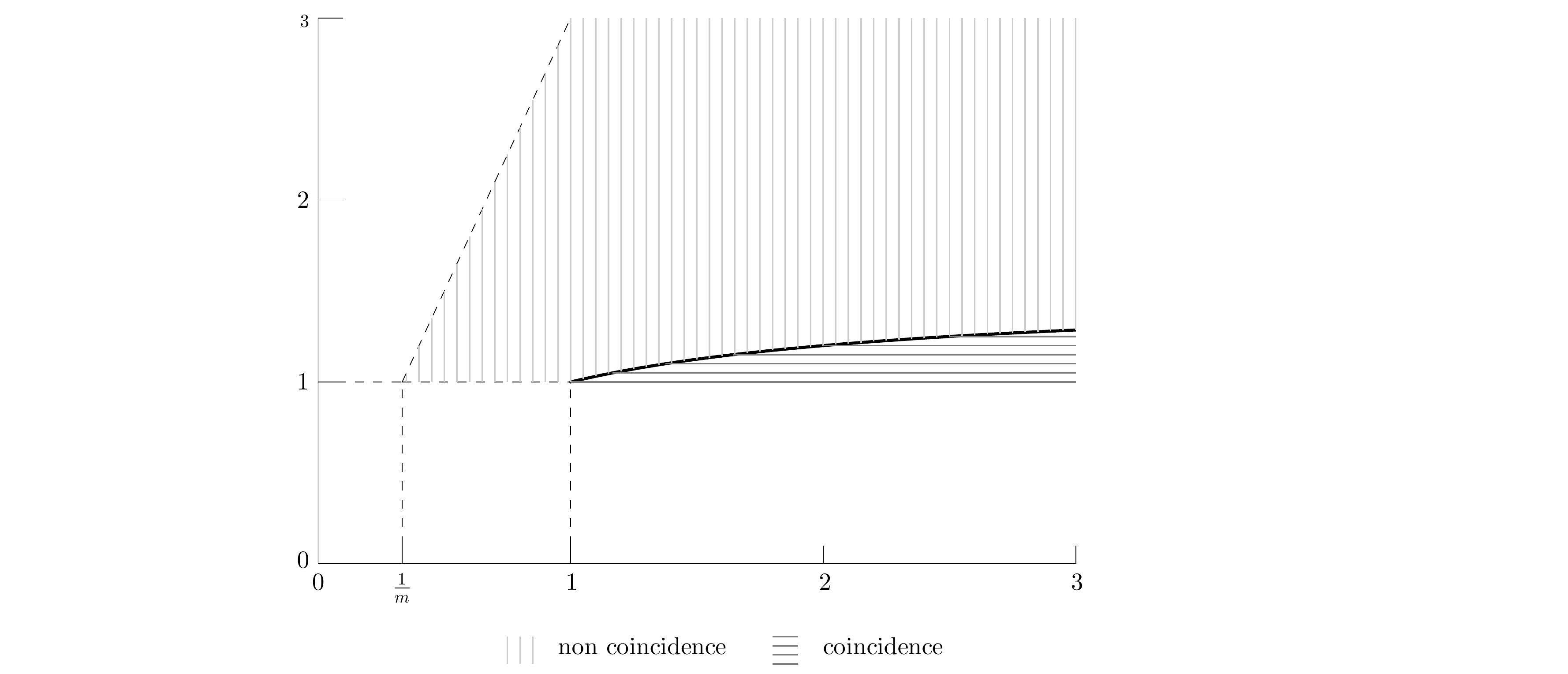}
\caption{ Coincidence zone for $\Pi_{\mathrm{as}(r;s)}(E_1,...,E_m;\mathbb{K}%
),(r,s)\in[1,\infty)\times[1,mr]$.}
\label{figure3}
\end{figure}

\bigskip

The table below details the results of coincidence and non coincidence in
the \textquotedblleft boundaries" of Figure \ref{figure3}. The only possible
open situation is the case $(r;s)$ with $s=1$ and $r<1$, which we answer in
the next proposition; the idea of the proof of this proposition is borrowed
from \cite{dimant}.

\bigskip

\begin{center}
\begin{tabular}{|c|c|c|}
\hline
$\frac{1}{m} \leq r<1$ & $s=1$ & not known \\ \hline
$r>\frac{1}{m}$ & $s=mr$ & non coincidence \\ \hline
$r\geq1$ & $s=1$ & coincidence \\ \hline
$r\geq1$ & $s=\frac{mr}{mr+1-r}$ & coincidence \\ \hline
\end{tabular}
\end{center}

\bigskip

\begin{proposition}
If $m\geq 1$ is a positive integer, then
\begin{equation*}
\inf \left\{ r:\mathcal{L}(E_{1},...,E_{m};\mathbb{K})=\Pi _{\mathrm{as}%
\left( r;1\right) }(E_{1},...,E_{m};\mathbb{K})\text{ for all
infinite--dimensional Banach spaces }E_{j}\right\} =1.
\end{equation*}
\end{proposition}

\begin{proof}
The equality holds for $r=1$; this is the so called Defant--Voigt theorem.
It remains to prove that the equality does not hold for $r<1$. This is
simple; we just need to choose $E_{j}=c_{0}$ for all $j$ and suppose that%
\begin{equation}
\mathcal{L}(E_{1},...,E_{m};\mathbb{K})=\Pi _{\mathrm{as}\left( r;1\right)
}(E_{1},...,E_{m};\mathbb{K}).  \label{u7l}
\end{equation}%
For all positive integers $n$, consider the $m$-linear forms%
\begin{equation*}
T_{n}:c_{0}\times \cdots \times c_{0}\rightarrow \mathbb{K}
\end{equation*}%
\bigskip defined by%
\begin{equation*}
T_{n}(x^{(1)},...,x^{(m)})=\sum\limits_{j=1}^{n}x_{j}^{(1)}\cdots
x_{j}^{(m)}.
\end{equation*}%
Then it is plain that $\left\Vert T_{n}\right\Vert =n$, and from (\ref{u7l})
there is a $C\geq 1$ such that
\begin{equation*}
\left( \sum\limits_{j=1}^{n}\left\vert T_{n}(e_{j},...,e_{j})\right\vert
^{r}\right) ^{\frac{1}{r}} \leq C\left\Vert T_{n}\right\Vert
\prod\limits_{k=1}^{m}\sup_{\varphi \in B_{E_{k}^{\ast
}}}\sum\limits_{j=1}^{n}\left\vert \varphi (e_{j})\right\vert =Cn,
\end{equation*}
i.e., $n^{1/r}\leq Cn$ and thus $r\geq 1$.
\end{proof}

\bigskip

This simple proposition ensures that the zone defined by $r<1$ and $s=1$ in
the Figure \ref{figure3} is a non coincidence zone, i.e., the Defant--Voigt
theorem is optimal. Therefore, we can make a new table for the results of
coincidence and non coincidence in the \textquotedblleft boundaries" of
Figure \ref{figure3}:

\bigskip

\begin{center}
\begin{tabular}{|c|c|c|}
\hline
$\frac{1}{m} \leq r<1$ & $s=1$ & non coincidence \\ \hline
$r\geq\frac{1}{m}$ & $s=mr$ & non coincidence \\ \hline
$r\geq1$ & $s=1$ & coincidence \\ \hline
$r\geq1$ & $s=\frac{mr}{mr+1-r}$ & coincidence \\ \hline
\end{tabular}
\end{center}


\bigskip

\begin{thebibliography}{99}
\bibitem{n} N. Albuquerque, F. Bayart, D. Pellegrino and J. Seoane--Sep\'{u}%
lveda, Optimal Hardy--Littlewood type inequalities for polynomials and
multilinear operators, to appear in Israel Journal of Mathematics (2014).

\bibitem{alb} N. Albuquerque, F. Bayart, D. Pellegrino and J. Seoane--Sep%
\'{u}lveda, Sharp generalizations of the multilinear Bohnenblust--Hille
inequality. J. Funct. Anal. \textbf{266} (2014), 3726--3740.



\bibitem{alencarmatos} R. Alencar and M. C. Matos, Some classes of
multilinear mappings between Banach spaces, Publ. Dep. An\'{a}lisis
Matematico, Universidad Complutense de Madrid, Section 1, n. \textbf{12},
1989.

\bibitem{aromlacruz} R. M. Aron, M. Lacruz, R. A. Ryan and A. M. Tonge, The
generalized Rademacher functions. Note Mat. \textbf{12} (1992), 15--25.

\bibitem{bay} F. Bayart, Maximum modulus of random polynomials. Q. J. Math.
63 (2012), no. 1, 21--39.

\bibitem{bernal} L. Bernal--Gonz\'{a}lez, D. Pellegrino and J. Seoane--Sep%
\'{u}lveda, Linear subsets of nonlinear sets in topological vector spaces,
Bull. Amer. Math. Soc. \textbf{51} (2014), 71--130.

\bibitem{bh} H. F. Bohnenblust and E. Hille, On the absolute convergence of
Dirichlet series. Ann. of Math. \textbf{32} (1931), 600--622.


\bibitem{studia} G. Botelho, D. Cariello, V. Favaro, D. Pellegrino and J. B.
Seoane--Sep\'{u}lveda, Distinguished subspaces of $L_{p}$ of maximal
dimension, Studia Math., \textbf{215} (3) (2013), 261-280.

\bibitem{campo} G. Botelho. J. Campos, On the transformation of
vector-valued sequences by multilinear operators,  \qquad arXiv:1410.4261.

\bibitem{REMC2010} G. Botelho, C. Michels and D. Pellegrino, Complex
interpolation and summability properties of multilinear operators. Rev. Mat.
Complut. \textbf{23} (2010), 139--161.






\bibitem{portmath} G. Botelho and D. Pellegrino, Coincidence situations for
absolutely summing non--linear mappings. Port. Math. (N.S.) \textbf{64}
(2007), n. 2, 175--191.

\bibitem{indag} G. Botelho, D. Pellegrino, P. Rueda, Summability and
estimates for polynomials and multilinear mappings. Indag. Math. (N.S.)
\textbf{19} (2008), n. 1, 23--31.

\bibitem{Di} J. Diestel, H. Jarchow and A. Tonge, Absolutely summing
operators, Cambridge University Press, Cambridge, 1995.

\bibitem{dimant} V. Dimant and P. Sevilla--Peris, Summation of coefficients
of polynomials on $\ell_{p}$ spaces, arXiv:1309.6063v1 [math.FA].

\bibitem{drew} L. Drewnowski, Quasicomplements in $F$--spaces, Studia Math.
\textbf{77} (1984) 373--391.

\bibitem{hardy} G. Hardy and J. E. Littlewood, Bilinear forms bounded in
space $[p, q]$, Quart. J. Math. \textbf{5} (1934).

\bibitem{KT} D. Kitson and R. M. Timoney, Operator ranges and spaceability,
J. Math. Anal. Appl. \textbf{378}, 2 (2011) 680--686.

\bibitem{LLL} J. E. Littlewood, On bounded bilinear forms in an infinite
number of variables. Quart. J. (Oxford Ser.) \textbf{1} (1930), 164--174.

\bibitem{matos} M. C. Matos, Fully absolutely summing mappings and Hilbert
Schmidt operators. Collect. Math. \textbf{54} (2003) 111--136.

\bibitem{matos2} M. C. Matos, On multilinear mappings of nuclear type. Rev.
Mat. Univ. Complut. Madrid \textbf{6} (1993), n. 1, 61--81.


\bibitem{pisier2} B. Maurey, G. Pisier, S\'{e}ries de variables al\'{e}%
atoires vectorielles ind\'{e}pendantes et propri\'{e}t\'{e}s g\'{e}om\'{e}%
triques des espaces de Banach. (French) Studia Math. 58 (1976), no. 1,
45--90.


\bibitem{laa} D. Pellegrino, Sharp coincidences for absolutely summing
multilinear operators. Linear Algebra and its Applications, \textbf{440}
(2014) 188--196.

\bibitem{per} D. P\'{e}rez--Garc\'{\i}a, Operadores multilineales
absolutamente sumantes, Thesis, Universidad Complutense de Madrid, 2003.

\bibitem{arc} D. P\'{e}rez--Garc\'{\i}a and I. Villanueva, Multiple summing
operators on $C(K)$ spaces. Ark. Mat. \textbf{42} (2004), 153--171.

\bibitem{piet} A. Pietsch, Absolut $p$--summierende Abbildungen in normieten
Raumen. Studia Math. \textbf{27} (1967), 333--353.

\bibitem{pi1} A. Pietsch, Ideals of multilinear functionals, Proceedings of
the Second International Conference on Operator Algebras, Ideals and Their
Applications in Theoretical Physics, Teubner--texte Math. \textbf{67}
(Teubner, Leipzig, 1983) 185--199.

\bibitem{pp} G. Pisier, Factorization of operators through $L_{p\infty}$ or $%
L_{p1}$ and non-commutative generalizations. Math. Ann. \textbf{276} (1986),
105--136.

\bibitem{pra} T. Praciano--Pereira, On bounded multilinear forms on a class
of $\ell _{p}$ spaces. J. Math. Anal. Appl. \textbf{81} (1981), n. 2,
561--568.

\bibitem{popa5} D. Popa, Mixing multilinear operators with or without a linear analogue. Integral Equations Operator Theory 75 (2013), no. 3, 323–339.

\bibitem{popa} D. Popa, Multiple summing, dominated and summing operators on
a product of $\ell _{1}$ spaces. Positivity 18 (2014), no. 4, 751--765.

\bibitem{popa2} D. Popa, Remarks on multiple summing operators on $C(K)$-spaces. Positivity 18 (2014), no. 1, 29--39.

\bibitem{rueda} P. Rueda, E.A. S\'{a}nchez-P\'{e}rez, Factorization of $p$%
-dominated polynomials through $L_{p}$-spaces. Michigan Math. J. 63 (2014),
no. 2, 345--353.

\bibitem{serrano} D.M. Serrano-Rodr\'{\i}guez, Absolutely $\gamma $-summing
multilinear operators. Linear Algebra Appl. 439 (2013), no. 12, 4110--4118.

\bibitem{zal} I. Zalduendo, Estimates for multilinear forms on $\ell_{p}$
spaces, Proc. Roy. Irish Acad. Sect. A 93 (1993), n. 1, 137--142.

\end{thebibliography}
\end{document}